\documentclass[12pt,reqno]{amsart}
\usepackage{amsmath, amsthm, amssymb, stmaryrd}

\advance \topmargin by -\headheight
\advance \topmargin by -\headsep
     
\evensidemargin \oddsidemargin
\newtheorem{theorem}{Theorem}[section]
\newtheorem{lemma}[theorem]{Lemma}

\theoremstyle{definition}

\theoremstyle{remark}

\numberwithin{equation}{section}

\def\bfn{{\mathbf n}}

\def\bfx{{\mathbf x}}
\def\bfy{{\mathbf y}}
\def\bfz{{\mathbf z}}

\def\dbC{{\mathbb C}}

\def\dbN{{\mathbb N}}  

\def\dbZ{{\mathbb Z}}\def\dbQ{{\mathbb Q}}

\def\grO{{\mathfrak O}}

\def\tet{{\theta}}

\def\sig{{\sigma}}

\def\eps{\varepsilon}

\def\le{\leqslant} \def\ge{\geqslant}

\begin{document}
\title[Shifted analogues of the divisor function]{A paucity problem associated with a shifted 
integer analogue of the divisor function}
\author[W. Heap]{Winston Heap}
\address{WH: Department of Mathematics, Shandong University, Jinan, Shandong 250100, China}
\email{winstonheap@gmail.com}
\author[A. Sahay]{Anurag Sahay}
\address{AS: Department of Mathematics, University of Rochester, 908 Hylan Building, P.O. 
Box 270138, Rochester, NY 14627, USA}
\email{asahay@ur.rochester.edu}
\author[T. D. Wooley]{Trevor D. Wooley}
\address{TDW: Department of Mathematics, Purdue University, 150 N. University Street, 
West Lafayette, IN 47907-2067, USA}
\email{twooley@purdue.edu}
\date{}
\dedicatory{}
\begin{abstract}Suppose that $\tet$ is irrational. Then almost all elements 
$\nu\in \dbZ[\tet]$ that may be written as a $k$-fold product of the shifted integers 
$n+\tet$ $(n\in \dbN)$ are thus represented essentially uniquely.
\end{abstract}
\maketitle

\section{Introduction} Given a complex number $\tet$, the shifted integer analogue of the 
natural numbers $\dbN+\tet=\{ n+\tet:n\in \dbN\}$ possesses, at a superficial level, 
many additive properties in common with its unshifted cousin $\dbN$. For multiplicative 
problems, the close connections plausible in the additive setting rapidly evaporate. In this 
note we examine a shifted analogue of the restricted divisor function. Thus, when 
$\nu \in \dbZ[\tet]$, we consider the function
\[
\tau_k(\nu;X,\tet)=\underset{(d_1+\tet)\cdots (d_k+\tet)=\nu}
{\sum_{1\le d_1\le X}\ldots \sum_{1\le d_k\le X}}1.
\]
The mean value
\[
\sum_{\nu \in \dbZ[\tet]}\tau_k(\nu;X,\tet)^2
\]
counts the number of integral solutions of the equation
\begin{equation}\label{1.1}
(x_1+\tet)\cdots (x_k+\tet)=(y_1+\tet)\cdots (y_k+\tet),
\end{equation}
with $1\le x_i,y_i\le X$ $(1\le i\le k)$. We show that when $\tet\not\in \dbQ$, then almost 
all solutions of (\ref{1.1}) are the diagonal ones in which $(x_1,\ldots ,x_k)$ is a 
permutation of $(y_1,\ldots ,y_k)$. Thus, almost all elements $\nu \in \dbZ[\tet]$ that 
may be written as a $k$-fold product of shifted integers $n+\tet$ $(n\in \dbN)$ are 
represented essentially uniquely in this manner.\par

In order to describe our conclusions in more detail, it is convenient to introduce some 
notation. Denote by $T_k(X)$ the number of $k$-tuples $\bfx$ and $\bfy$ in which 
$1\le x_i,y_i\le X$ $(1\le i\le k)$, and $(x_1,\ldots ,x_k)$ is a permutation of 
$(y_1,\ldots ,y_k)$. Thus $T_k(X)=k!X^k+O(X^{k-1})$. We begin with the simplest 
situation in which $\tet \in \dbC$ is either transcendental, or else algebraic of large degree 
over $\dbQ$.

\begin{theorem}\label{theorem1.1} Let $k\in \dbN$ and $\eps>0$. Suppose that 
$\tet\in \dbC$ is either transcendental, or else algebraic of degree $d$ over $\dbQ$, 
where $d\ge k$. Then one has
\[
\sum_{\nu\in \dbZ[\tet]}\tau_k(\nu;X,\tet)^2=T_k(X).
\]
\end{theorem}

The situation in which $\tet$ is algebraic of small degree is more complicated.

\begin{theorem}\label{theorem1.2} Let $k\in \dbN$ and $\eps>0$. Suppose that 
$\tet\in \dbC$ is algebraic of degree $d$ over $\dbQ$, where $2\le d<k$. Then one has
\[
\sum_{\nu\in \dbZ[\tet]}\tau_k(\nu;X,\tet)^2=T_k(X)+O(X^{k-d+1+\eps}).
\]
Here, the implicit constant in Landau's notation may depend on $k$, $\eps$ and $\tet$.
\end{theorem}

It follows that when $\tet\not\in \dbQ$, then there is a paucity of non-diagonal solutions 
in the equation (\ref{1.1}). Moreover, one has the asymptotic formula
\[
\sum_{\nu\in \dbZ[\tet]}\tau_k(\nu;X,\tet)^2=k!X^k+O(X^{k-1+\eps}).
\]
These conclusions are in marked contrast with the corresponding situation in which 
$\tet\in \dbQ$. When $\tet$ is rational, experts will recognise that a straightforward 
exercise employing the circle method yields the lower bound
\[
\sum_{\nu\in \dbQ}\tau_k(\nu;X,\tet)^2\gg_{\tet,k} X^k(\log X)^{(k-1)^2}.
\]
Indeed, additional work would exhibit an asymptotic formula in place of this lower bound. In 
this regard, we note that the contour integral methods of \cite{HNR, HL} would also be 
accessible. The inquisitive reader interested in paucity problems for diagonal Diophantine 
systems will find a representative slice of the relevant literature in 
\cite{Sal2007, SW1997, VW1997, Woo2021}.\par

One motivation for considering this problem is that such multiplicative equations arise 
naturally when studying the higher moments of zeta and $L$-functions. In particular, 
equation \eqref{1.1} is intimately related to the  moments and value distribution of the 
Hurwitz zeta function $\zeta(s,\tet)$ with shift parameter $0<\tet \le 1$. For irrational shifts $\tet$, this forms the focus of an ongoing project of the first and second author. For rational shifts $\tet$, see \cite{Sah2021}.\par 

Perhaps it is worth stressing that the equation (\ref{1.1}) corresponds to a system of 
polynomial equations with integral variables. In order to illustrate this point, consider the 
situation in which $k=3$ and $\tet=\sqrt{2}$. We make use of the linear independence of 
$1$ and $\sqrt{2}$ over $\dbQ$. On noting that
\[
(x_1+\sqrt{2})(x_2+\sqrt{2})(x_3+\sqrt{2})=x_1x_2x_3+2(x_1+x_2+x_3)+\sqrt{2}
(x_1x_2+x_2x_3+x_3x_1+2),
\]
we find that the equation (\ref{1.1}) holds if and only if
\begin{equation}\label{1.2}
\begin{aligned}x_1x_2x_3+2(x_1+x_2+x_3)&=y_1y_2y_3+2(y_1+y_2+y_3)\\
x_1x_2+x_2x_3+x_3x_1&=y_1y_2+y_2y_3+y_3y_1.\end{aligned}
\end{equation}
In this scenario, we conclude from Theorem \ref{theorem1.2} that the number $N(X)$ of 
integral solutions of the system (\ref{1.2}) with $1\le x_i,y_i\le X$ $(1\le i\le 3)$ satisfies
\[
N(X)=6X^3+O(X^{2+\eps}).
\]

\par The basic strategy that we employ in the proofs of Theorems \ref{theorem1.1} and 
\ref{theorem1.2} is based on the generation of multiplicative polynomial identities. These 
are inspired by an examination of the polynomial
\[
\prod_{i=1}^k(t-x_i)-\prod_{i=1}^k(t-y_i).
\]
There are parallels here with the third author's treatment of the Vinogradov system in joint 
work with Vaughan \cite{VW1997}. We also interpret the function $\tau_k(\nu;X,\tet)$ as 
the number of integral solutions of the equation
\begin{equation}\label{1.3}
(x_1+\tet)(x_2+\tet)\cdots (x_k+\tet)=\nu ,
\end{equation}
with $1\le x_i\le X$ $(1\le i\le k)$. This can be seen as a restriction of the $k$-fold divisor 
function in the ring of integers $\grO_K$ of the number field $K=\dbQ(\tet)$. Immediate 
appeal to such ideas is limited by the observation that the divisors occurring on the left 
hand side of (\ref{1.3}) come from a thin subset of all the algebraic integers in $\grO_K$. 
Nonetheless, a crude bound $\tau_k(\nu;X,\tet)=O(X^\eps)$ stemming from such ideas 
plays a role in the concluding phase of our proof of Theorem \ref{theorem1.2}.\par

Our basic parameter is $X$, a sufficiently large positive number. Whenever $\eps$ 
appears in a statement, either implicitly or explicitly, we assert that the statement holds 
for each $\eps>0$. In this paper, implicit constants in the notations of Landau and 
Vinogradov may depend on $\eps$, $k$, and $\tet$. We make frequent use of vector 
notation in the form $\bfx=(x_1,\ldots,x_k)$. Here, the dimension $k$ is permitted to 
depend on the course of the argument.\par
 
\noindent {\bf Acknowledgements:} The second author is grateful to his PhD advisor, Steve 
Gonek, for support and encouragement. The third author's work is supported by NSF grants 
DMS-2001549 and DMS-1854398. The second and third authors also benefitted from 
activities hosted by the American Institute of Mathematics, San Jose supported via the 
latter grant. 

\section{The proof of Theorem \ref{theorem1.1}} We begin in this section by considering 
the situation in which $\tet\in \dbC$ is either transcendental, or else is algebraic of degree 
$d\ge k$ over $\dbQ$. In such circumstances, we rewrite the equation (\ref{1.1}) by using 
elementary symmetric polynomials $\sig_j(\bfz)\in \dbZ[z_1,\ldots ,z_k]$. These may be 
defined for $j\ge 0$ by means of the generating function identity
\[
\sum_{j=0}^k\sig_j(\bfz)t^{k-j}=\prod_{i=1}^k(t+z_i).
\]
The equation (\ref{1.1}) may thus be rewritten in the form
\[
\sum_{j=0}^k\sig_j(\bfx)\tet^{k-j}=\sum_{j=0}^k\sig_j(\bfy)\tet^{k-j}.
\]
Since $\sig_0(\bfx)=1=\sig_0(\bfy)$, we find that
\begin{equation}\label{2.1}
\sum_{j=1}^k(\sig_j(\bfx)-\sig_j(\bfy))\tet^{k-j}=0.
\end{equation}
In our present situation with $\tet$ either transcendental, or else algebraic of degree 
$d\ge k$ over $\dbQ$, the complex numbers $1,\tet,\ldots ,\tet^{k-1}$ are linearly 
independent over $\dbQ$. Then it follows from (\ref{2.1}) that $\sig_j(\bfx)=\sig_j(\bfy)$ 
$(1\le j\le k)$. In particular, one obtains the polynomial identity
\begin{equation}\label{2.2}
\prod_{j=1}^k(t-x_j)=\prod_{j=1}^k(t-y_j).
\end{equation}

\par The polynomial relation (\ref{2.2}) implies that left and right hand sides must have the 
same zeros with identical multiplicities. Hence $(x_1,\ldots ,x_k)$ must be a permutation of 
$(y_1,\ldots ,y_k)$. The conclusion
\[
\sum_{\nu\in \dbZ[\tet]}\tau_k(\nu;X,\tet)^2=T_k(X)
\]
is then immediate on considering the Diophantine interpretation (\ref{1.1}) of the mean 
value on the left hand side. This completes the proof of Theorem \ref{theorem1.1}.

\section{The proof of Theorem \ref{theorem1.2}} We now assume that $\tet\in \dbC$ is 
an algebraic number of degree $d$ over $\dbQ$, with $2\le d<k$. In this situation the 
equation (\ref{1.1}) simplifies, since $\tet^d$ may be expressed as a $\dbQ$-linear 
combination of $1,\tet,\ldots ,\tet^{d-1}$. However, the equation (\ref{2.1}) no longer 
delivers $k$ independent polynomial equations, but instead $d$ such equations with $d<k$. 
The strategy of \S2 is thus no longer applicable.\par

Let $\bfx,\bfy$ be an integral solution of the equation (\ref{1.1}) with $1\le x_i,y_i\le X$ 
$(1\le i\le k)$, in which $(x_1,\ldots ,x_k)$ is not a permutation of $(y_1,\ldots ,y_k)$. 
Observe first that if $x_i=y_j$ for any indices $i$ and $j$ with $1\le i,j\le k$, then we may 
cancel the factors $x_i+\tet$ and $y_j+\tet$, respectively, from the left and right hand 
sides of (\ref{1.1}). It thus suffices to establish the conclusion of Theorem 
\ref{theorem1.2} with $k$ replaced by $k-1$. Here, of course, if $d\ge k-1$, then the 
desired conclusion follows from Theorem \ref{theorem1.1}. By repeatedly cancelling pairs 
of equal factors in this way, it is apparent that there is no loss of generality in supposing 
henceforth that $x_i=y_j$ for no indices $i$ and $j$ with $1\le i,j\le k$.\par

Consider the polynomial
\begin{equation}\label{3.1}
F(t)=\prod_{i=1}^k(t+x_i)-\prod_{i=1}^k(t+y_i).
\end{equation}
This polynomial has degree at most $k-1$, and so for suitable integers $a_j=a_j(\bfx,\bfy)$ 
$(0\le j\le k-1)$, we may write
\[
F(t)=a_0+a_1t+\ldots +a_{k-1}t^{k-1}.
\]
Note that for $0\le j\le k-1$, one has
\begin{equation}\label{3.2}
|a_j|=|\sig_{k-j}(\bfx)-\sig_{k-j}(\bfy)|\ll X^{k-j}.
\end{equation}
Next, denote by $m_\tet\in \dbZ[t]$ the minimal polynomial of $\tet$ over $\dbZ$. Then 
$m_\tet$ is irreducible of degree $d$ over $\dbZ$, and if $m_\tet$ has leading coefficient 
$c_d\ne 0$, then $c_d^{-1}m_\tet\in \dbQ[t]$ is the usual minimal polynomial of $\tet$ 
over $\dbQ$. We may write
\[
m_\tet(t)=c_0+c_1t+\ldots +c_dt^d,
\]
in which $|c_j|\ll_\tet 1$ $(0\le j\le d)$. We observe from (\ref{1.1}) and (\ref{3.1}) that
\[
F(\tet)=\prod_{i=1}^k(x_i+\tet)-\prod_{i=1}^k(y_i+\tet)=0,
\]
whence $m_\tet(t)$ divides $F(t)$. Consequently, there is a polynomial 
$\Psi(t)=\Psi_\tet(t;\bfx,\bfy)\in \dbZ[t]$ having the property that
\begin{equation}\label{3.3}
F(t)=m_\tet(t)\Psi(t).
\end{equation}
Since $\deg(\Psi)=\deg(F)-\deg(m_\tet)\le k-1-d$, we may write
\[
\Psi(t)=b_0+b_1t+\ldots +b_{k-1-d}t^{k-1-d},
\]
where $b_m\in \dbZ$ $(0\le m\le k-1-d)$. Our immediate goal is to bound the coefficients 
$b_m$.\par

We claim that for $0\le m\le k-d-1$, one has
\begin{equation}\label{3.4}
|b_m|\ll X^{k-d-m}.
\end{equation}
This we establish by considering the formal Laurent series for $m_\tet(t)^{-1}$. Thus, we 
have $m_\tet(t)^{-1}=e(t)\in \dbQ((1/t))$, where for suitable rational coefficients 
$e_j\in \dbQ$ $(j\ge d)$ one has
\[
e(t)=\sum_{j=d}^\infty e_jt^{-j}=\frac{1}{c_dt^d}(1+c_d^{-1}c_{d-1}t^{-1}+\ldots 
+c_d^{-1}c_0t^{-d})^{-1}.
\]
Note here that $c_d\ne 0$, and further that $e_j\ll_{\tet,j}1$. We may therefore infer from 
(\ref{3.3}) that $\Psi(t)=e(t)F(t)$, whence
\[
\sum_{m=0}^{k-1-d}b_mt^m=\Biggl( \sum_{j=d}^\infty e_jt^{-j}\Biggr) \Biggl( 
\sum_{i=0}^{k-1}a_it^i\Biggr) .
\]
In view of the bound (\ref{3.2}), we deduce that for $0\le m\le k-1-d$ one has
\begin{align*}
b_m&=e_da_{m+d}+e_{d+1}a_{m+d+1}+\ldots +e_{k-1-m}a_{k-1}\\
&\ll X^{k-d-m}+X^{k-d-m-1}+\ldots +X\ll X^{k-d-m}.
\end{align*}
This confirms the bound (\ref{3.4}). We may suppose henceforth that there is a positive 
number $C=C(k,\tet)$ having the property that
\begin{equation}\label{3.5}
|b_m|\le CX^{k-d-m}\quad (0\le m\le k-d-1).
\end{equation}

\par We now arrive at the polynomial identity that does the heavy lifting in the proof of 
Theorem \ref{theorem1.2}.

\begin{lemma}\label{lemma3.1} Suppose that $\bfx,\bfy$ is an integral solution of the 
equation (\ref{1.1}) with $1\le x_i,y_i\le X$ $(1\le i\le k)$, in which $x_i=y_j$ for no indices 
$i$ and $j$ with $1\le i,j\le k$. Then, for each index $j$ with $1\le j\le k$, there is an 
integer $\rho_j$, with $1\le |\rho_j|\le kCX^{k-d}$, having the property that
\[
\prod_{i=1}^k(x_i-y_j)=\rho_jm_\tet(-y_j).
\]
\end{lemma}

\begin{proof} Recalling (\ref{3.1}) and (\ref{3.3}), we see that
\[
F(-y_j)=\prod_{i=1}^k(x_i-y_j)=m_\tet(-y_j)\Psi(-y_j).
\]
But in view of (\ref{3.5}), one has
\[
|\Psi(-y_j)|\le \sum_{m=0}^{k-1-d}|b_m|y_j^m\le (k-d)CX^{k-d}.
\]
Thus, there is an integer $\rho_j=\Psi(-y_j)$ with $|\rho_j|\le kCX^{k-d}$ for which
\[
\prod_{i=1}^k(x_i-y_j)=m_\tet(-y_j)\rho_j.
\]
Notice here that since the left hand side is a non-zero integer, then so too are both factors 
on the right hand side. The conclusion of the lemma follows.
\end{proof}

We may now complete the proof of Theorem \ref{theorem1.2}. Our previous discussion 
ensures that it is sufficient to count solutions $\bfx,\bfy$ of (\ref{1.1}) with 
$1\le x_i,y_i\le X$ $(1\le i\le k)$, in which $x_i=y_j$ for no indices $i$ and $j$ with 
$1\le i,j\le k$. Given any such solution, an application of Lemma \ref{lemma3.1} with $j=k$ 
shows that, for some integer $\rho_k$ with $1\le |\rho_k|\le kCX^{k-d}$, one has
\begin{equation}\label{3.6}
\prod_{i=1}^k(x_i-y_k)=\rho_km_\tet(-y_k).
\end{equation}
Fix any one of the $O(X)$ possible choices for $y_k$, and likewise any one of the 
$O(X^{k-d})$ possible choices for $\rho_k$. Then we see from (\ref{3.6}) that each of the 
factors $x_i-y_k$ $(1\le i\le k)$ must be a divisor of the non-zero integer 
$N=\rho_km_\tet(-y_k)$. It therefore follows from an elementary estimate for the divisor 
function that there are $O(N^\eps)$ possible choices for $x_i-y_k$ $(1\le i\le k)$. Fix any 
one such choice. Then since $y_k$ has already been fixed, we see that $x_1,\ldots ,x_k$ 
and $y_k$ are now all fixed.\par

At this point we return to the equation (\ref{1.1}). By taking norms from $\dbQ(\tet)$ 
down to $\dbQ$, we see that
\[
\prod_{i=1}^km_\tet(-y_i)=\prod_{i=1}^km_\tet(-x_i).
\]
The right hand side here is already fixed and non-zero. A divisor function estimate 
therefore shows that there are $O(X^\eps)$ possible choices for integers $n_1,\ldots ,n_k$ 
having the property that
\[
m_\tet(-y_i)=n_i\quad (1\le i\le k).
\]
Fixing any one such choice for the $k$-tuple $\bfn$, we find that when $1\le i\le k$, there 
are at most $d$ choices for the integer solution $y_i$ of the polynomial equation 
$m_\tet(-t)=n_i$. Altogether then, the number of possible choices for $\bfx$ and $\bfy$ 
given a fixed choice for $y_k$ and $\rho_k$ is $O((NX)^\eps)$. Thus we conclude that the 
total number of possible choices for $\bfx$ and $\bfy$ is $O(X^{k-d+1+\eps})$, and hence
\[
\sum_{\nu \in \dbZ[\tet]}\tau_k(\nu;X,\tet)^2-T_k(X)\ll X^{k-d+1+\eps}.
\]
This completes the proof of Theorem \ref{theorem1.2}.

\bibliographystyle{amsbracket}

\begin{thebibliography}{18}

\bibitem{HNR} A. J. Harper, A. Nikeghbali and M. Radziwi\l\l, \emph{A note on Helson's 
conjecture on moments of random multiplicative functions}, in Analytic Number Theory 
(Springer, Cham, 2015), 145--169.

\bibitem{HL}W. P. Heap and S. Lindqvist, \emph{Moments of random multiplicative 
functions and truncated characteristic polynomials}, Q. J. Math. \textbf{67} (2016), no. 4, 
683--714.

\bibitem{Sah2021} A. Sahay, \emph{Moments of the Hurwitz zeta function on the critical 
line}, submitted, arXiv:2103.13542. 

\bibitem{Sal2007} P. Salberger, \emph{Rational points of bounded height on threefolds}, 
Analytic Number Theory, Clay Math. Proc. \textbf{7} (2007), pp. 207--216, Amer. Math. 
Soc., Providence, RI.

\bibitem{SW1997} C. M. Skinner and T. D. Wooley, \emph{On the paucity of non-diagonal 
solutions in certain diagonal Diophantine systems}, Quart. J. Math. Oxford (2) \textbf{48} 
(1997), 255--277.

\bibitem{VW1997} R. C. Vaughan and T. D. Wooley, \emph{A special case of Vinogradov's 
mean value theorem}, Acta Arith. \textbf{79} (1997), no. 3, 193--204.

\bibitem{Woo2021} T. D. Wooley, \emph{Paucity problems and some relatives of 
Vinogradov's mean value theorem}, submitted, arXiv:2107.12238. 

\end{thebibliography}
\providecommand{\bysame}{\leavevmode\hbox to3em{\hrulefill}\thinspace}

\end{document}